\documentclass[12pt,oneside]{amsart}
\usepackage{calc,geometry, verbatim, graphicx, amssymb, color, mathabx, mathtools, enumitem, bm, mathrsfs, amsmath, amsthm, ulem, xspace, mathtools, multirow, array, adjustbox, physics, bm}
\usepackage[usenames,dvipsnames,svgnames,table]{xcolor}
\usepackage[pdftex,bookmarks,colorlinks,breaklinks]{hyperref}  
\hypersetup{linkcolor=blue,citecolor=red,filecolor=dullmagenta,urlcolor=darkblue} 
\newtheorem{theorem}{Theorem}[section]
\newtheorem{corollary}[theorem]{Corollary}

\newtheorem{conjecture}[theorem]{Conjecture}
\newtheorem{remark}[theorem]{Remark}
\newtheorem{example}[theorem]{Example}
\newtheorem{question}[theorem]{Problem} 
\newcommand{\Mod}[1]{\ (\mathrm{mod}\ #1)}

\newcommand\mult{\operatorname{\textup{{\fontfamily{ptm}\selectfont mult}}}}
\newcommand\dg{\operatorname{\textup{{\fontfamily{ptm}\selectfont deg}}}}

\newcommand{\cc}{\mathcal}

\newcommand\roundup[1]{\left\lceil#1\right\rceil}
\newcommand\rounddown[1]{\left\lfloor#1\right\rfloor}

      \makeatletter
      \def\@setcopyright{}
      \def\serieslogo@{}
      \makeatother      
\begin{document}
   \author{Amin Bahmanian}
   \address{Department of Mathematics,
  Illinois State University, Normal, IL USA 61790-4520}
   \author{Anna Johnsen}
  \address{Department of Mathematics,
  Illinois State University, Normal, IL USA 61790-4520}
  \email{anna.johnsen@gmail.com}

\title[Embedding Irregular Colorings into Connected Factorizations]{Embedding Irregular Colorings into Connected Factorizations}

   \begin{abstract}  
For $\vb r:=(r_1,\dots,r_k)$, an {\it $\vb r$-factorization} of the complete $\lambda$-fold $h$-uniform $n$-vertex hypergraph $\lambda K_n^h$ is a partition of (the edges of) $\lambda K_n^h$ into $F_1,\dots, F_k$ such that for $i=1,\dots,k$,  $F_i$ is $r_i$-regular and spanning. Suppose that $n \geq (h-1)(2m-1)$. Given a partial $\vb r$-factorization of $\lambda K_m^h$, that is, a coloring (i.e. partition) $P$ of the edges of  $\lambda K_m^h$ into $F_1,\dots, F_k$ such that for $i=1,\dots,k$,  $F_i$ is  spanning and the degree of each vertex in $F_i$ is at most $r_i$, we find necessary and sufficient conditions that ensure $P$ can be extended to a connected $\vb r$-factorization of $\lambda K_n^h$ (i.e.  an $\vb r$-factorization in which each factor is connected). Moreover, we  prove a general result that implies the following. Given a partial $\vb s$-factorization $P$ of {\it any} sub-hypergraph of  $\lambda K_m^h$, where $\vb s:=(s_1,\dots,s_q)$ and $q$ is not too big, we find necessary and sufficient conditions under which $P$ can be embedded into  a connected $\vb r$-factorization of $\lambda K_n^h$. These results can be seen as  unified generalizations of various classical combinatorial results such as Cruse's theorem on embedding partial symmetric latin squares, Baranyai's theorem on factorization of hypergraphs, Hilton's theorem on extending path decompositions into Hamiltonian decompositions, H\"{a}ggkvist and Hellgren's theorem on extending 1-factorizations, and Hilton, Johnson, Rodger, and Wantland's theorem on embedding connected factorizations.

   \end{abstract}
   \date{\today}

   \maketitle

\section{Background and  Statement of the Main Results} 
Nearly seventy years ago, Ryser \cite{MR42361} found conditions that ensure  an $r\times s$ latin rectangle can be embedded into an $n\times n$ latin square.  Cruse \cite{MR0329925} provided a symmetric analogue of Ryser's theorem by finding conditions under which an  $r\times r$ symmetric latin rectangle can be embedded into an $n\times n$ symmetric  latin square. In graph theoretic terms, these two results are equivalent to finding conditions under which a proper edge-coloring of the complete bipartite graph $K_{r,s}$ and  the complete graph $\mathbb K_r$ (this is the complete graph $K_r$ with a loop on each vertex) can be extended to a one-factorization of $K_{n,n}$ and  $\mathbb K_n$, respectively.  
The goal of this paper is to explore  higher dimensional analogues of such classical results. To elaborate, let us provide some definitions first.

Let $\cc G$ be a hypergraph with vertex set $V$ and edge (multi)set $E$. We allow edges of $\cc G$ to contain multiple copies of any vertex $v\in V$. If  every vertex  degree in $\cc G$ is exactly  $r$, then $\cc G$ is said to be {\it  $r$-regular}, and if none of the components of $\cc G$ is $r$-regular, then $\cc G$ is said to be {\it $r$-irregular}.  A {\it  $k$-coloring} of $\cc G$ is a mapping $f: E \rightarrow [k]$ where $[k]:=\{1,\dots, k\}$. Any $k$-coloring is also a $(k+\ell)$-coloring for $\ell\geq 1$. For $\vb r=(r_1,\dots,r_k)$, an {\it $\vb r$-factorization} of $\cc G$ is a $k$-coloring of $\cc G$ where for $i\in [k]$, color class  $i$, written $\cc G(i)$, induces  an $r_i$-regular spanning sub--hypergraph of $\cc G$; that is, an {\it $r_i$-factor} of $\cc G$. We say that a hypergraph $G$ is {\it connected} if for any two vertices $v, w \in V$, there exists a sequence of edges $e_1, e_2, ..., e_s$ in the edge set of $G$ such that $v\in e_1, w\in e_s$, and $e_i\cap e_j \neq \emptyset$ for all $i,j\in[s]$. 

Let $\lambda K_n^h$ be the $n$-vertex hypergraph whose edge multiset $E$ is  the collection of all $h$-subsets of the vertex set, each $h$-subset occurring exactly $\lambda$ times in $E$. Given a 1-factorization of $K_n^h$, if you think of the set $V$ of vertices  as the set of points, the set $E$ of edges as the set of lines, and 1-factors  as parallel classes, then  for every point $v \in V$ and for each line $\ell$ in $E$, there is exactly another line $\ell'$ which is parallel to $\ell$ (that is, contained in the same parallel class as $\ell$) and contains $v$. Hence, a 1-factorization is sometimes called  a {\it parallelism}.

For $\vb r:=(r_1,\dots,r_k)$, a {\it partial $\vb r$-factorization} of $\cc G$ is a $k$-coloring of $\cc G$ where for $i\in [k]$, color class  $i$, $\cc G(i)$,  induces a spanning sub-hypergraph of $\cc G$ in which the degree of each vertex is at most $r_i$. Whenever  $\vb r=(r,\dots,r)$, we replace $\vb r$ by $r$; for example, an $r$-factorization  is an  $\vb r$-factorization with $\vb r=(r,\dots,r)$. A partial 1-factorization is often called a {\it proper coloring}. 
Suppose that a partial $\vb r$-factorization $P$ of $\lambda K_m^h$ is extended to an $\vb r$-factorization of $\lambda K_n^h$. For $i\in[k]$, the existence of an $r_i$-factor in $\lambda K_n^h$ implies that $h \divides r_in$. Since the degree of each vertex in $\lambda K_n^h$ is $\lambda \binom{n-1}{h-1}$, in order to extend $P$ to an $\vb r$-factorization of $\lambda K_n^h$, the following conditions are necessary.
\begin{align} \label{divcond+h1}
    &&
    h \mid r_i n \quad \forall i\in[k],
    &&
    \sum_{i=1}^k r_i = \lambda\binom{n-1}{h-1}.
    &&
  \end{align} 
We  show that as long as $n$ satisfies a reasonable lower bound, these obvious necessary (divisibility) conditions are  sufficient.  A quadruple $(n,h,\lambda, \vb r)$ is {\it admissible} if it satisfies  \eqref{divcond+h1}.  Here is our first result which settles a conjecture of \cite{MR3915188}.
\begin{theorem}\label{thm:general}
 For $n \geq (h-1)(2m-1)$, a partial $\vb r$-factorization of $\lambda K_m^h$ can be extended to an $\vb r$-factorization of $\lambda K_n^h$ if and only if  $(n,h,\lambda, \vb r)$ is admissible. 
\end{theorem}
Since we know little about the structure of the given partial $\vb r$-factorization $P$ of $\lambda K_m^h$ (except for the maximum degree in each color class), it seems quite difficult to find an exact bound for $n$. However, if we put some restrictions on $P$, we may be able to obtain a better bound. A natural restriction to put on $P$ is to assume that  each color class of $P$ is regular. A notable example is the underappreciated  result of H\"{a}ggkvist and Hellgren \cite{MR1249714} proving that a 1-factorization of $K_m^h$ can be embedded into a $1$-factorization of $K_n^h$ if and only if $h$ divides $m$ and $n$, and $n\geq 2m$. This elegant result settled a conjecture of Baranyai and Brouwer \cite{BaranBrouwer77}. Bahmanian and Newman \cite{MR3910877} showed that whenever $\gcd (m,n,h)=\gcd (m,h)$,  an $r$-factorization of $K_m^h$ can be extended to an $r$-factorization of $K_n^h$ if and only if  $n\geq 2m$ and the obvious necessary divisibility conditions are satisfied.  
An implication of regularity of $P$ is that  some colors do not appear in $P$.  Our next result is in  this spirit, but it does not require regularity, and most importantly, it allows many edges of $\lambda K_m^h$ to be uncolored.
\begin{theorem} \label{incompembthm}
Let  $n \geq (h-1)(2m-1)$,  $\vb{s}=(s_1,\dots,s_q), \vb{r}=(r_1,\dots,r_k)$ such that $s_i\leq r_i$ for $1\leq i\leq q\leq k$, and 
\begin{align*} 
    \sum_{i=1}^q \rounddown{\frac{(r_i-s_i)m}{h}}+\sum_{i=q+1}^k \rounddown{\frac{r_im}{h}} \geq \lambda\binom{m}{h}.
\end{align*}
A partial $\vb{s}$-factorization of $\cc G\subseteq\lambda K_m^h$ can be embedded into an $\vb r$-factorization of $\lambda K_n^h$ if and only if $(n,h,\lambda, \vb r)$ is admissible.
\end{theorem} 
Theorem \ref{incompembthm} is  reminiscent  of Evans' theorem \cite{MR122728} that a partial 1-factorization of $F\subseteq K_{m,m}$ using $m$ colors can be extended to a 1-factorization of $K_{n,n}$ whenever $n\geq 2m$.
(In other words, for $n\geq 2m$, an incomplete latin square of order $m$ on $m$ symbols can be embedded in a latin square of order $n$). We remark that Theorem \ref{incompembthm} is new even if we restrict ourselves to the case $h=2$. 

Constructing 1-factorizations dates back to the  18th century (see for example Walecki's construction \cite{MR2394738}  and Sylvester's problem \cite{MR0416986}), but perhaps the first result on embedding factorizations is Ryser's theorem \cite{MR42361}, which sparked research in a wide range of problems in non-associative algebra and design theory \cite{MR1275861}. 
 Using  polyhedral combinatorics, Marcotte and Seymour \cite{MR1073098}  found  conditions that ensure a partial 1-factorization  of a subgraph of a  multiforest $G$ can be extended to  a partial 1-factorization of $G$, which led to a different line of work on embedding partial 1-factorizations \cite{MR4174122, MR4118383, MR1958148, MR3829287, MR4009303, MR3862954}. 

Last but not least, there has been an interest in finding factorizations that meet additional criterion. 
An early example of this  is due to Hilton, who settled conditions that ensure a partial 2-factorization of $K_m$ can be extended to a connected 2-factorization of $K_n$ \cite{MR746544}. Various extensions of this, including those with higher edge-connectivity conditions, can be found in \cite{MR1983358, MR2325799, MR1315436}. Let $m$ be the least common multiple of $h$ and $n$ and let $a=m/h$. Define the set of edges
$$\mathscr K=\{ \{1,\dots,h\}, \{h+1,\dots,2h\},\dots, \{(a-1)h+1,(a-1)h+2,\dots,ah\}\},$$ 
where the elements of the edges are considered mod $n$. The families obtained from $\mathscr K$ by permuting the elements of the underlying set $\{n\}$ are called {\it wreaths}.  A notoriously difficult conjecture of Baranyai and Katona asks for a wreath decomposition of $K_n^h$ \cite{MR1180500}. Observing that wreaths are connected factors (among other things), a small step toward settling this conjecture was to find connected factorizations of $\lambda K_n^h$  \cite{MR3213845}. Our next two results can be used to embed partial wreath decompositions into connected factorizations. 
\begin{theorem} \label{connthmcomb}
In Theorem \ref{thm:general},  for  $i\in [k]$, the $r_i$-factor  $\lambda K_n^h(i)$ is connected if and only if $r_i\geq 2$ and $\lambda K_m^h(i)$ is $r_i$-irregular.
\end{theorem}
\begin{theorem} \label{connthmcomb2} 
In Theorem \ref{incompembthm},  let $A\subseteq \{i \in [k]\ |\ r_i\geq 2\}, B=\{i\in [q] \ |\ r_i\neq s_i\}$, and define $\overline {r}_i=r_i-1$ if $i\in A$, and  $\overline {r}_i=r_i$ if $i\in [k]\backslash A$. If 
\begin{align*} 
    \sum_{i\in B} \rounddown{\frac{(\overline {r}_i-s_i)m}{h}}+\sum_{i\in [k]\backslash [q]} \rounddown{\frac{\overline {r}_i m}{h}} \geq \lambda\binom{m}{h},
\end{align*}
then for  $i\in A$, the $r_i$-factor   $\lambda K_n^h(i)$ is connected if and only if $\cc G(i)$ is $r_i$-irregular.
\end{theorem}

Before we prove Theorems \ref{thm:general}--\ref{connthmcomb2},  we provide several simple corollaries in the next section. These corollaries extend various classical combinatorial results. We postpone more involved applications to Section \ref{compappli}. When needed, we will  use  Bernoulli's inequality ($\forall p\geq 1, \forall x\geq -1: (1+x)^p\geq 1+px$) and the following combinatorial identities without further explanation. For  $n\geq m\geq h$
\begin{align*}
&\sum_{i=0}^h\binom{m}{i}\binom{n-m}{h-i}=\binom{n}{h}, 
&\sum_{i=1}^{h-1} i\binom{m}{i}\binom{n-m}{h-i}=m\left [\binom{n-1}{h-1}-\binom{m-1}{h-1}\right ].
\end{align*}

\section{Corollaries}
To exhibit the effectiveness of our main results, first we provide a few simple applications of Theorems \ref{thm:general} and \ref{connthmcomb}. 

\begin{corollary}
Let 
\begin{align*}
    && n &\geq (h-1)(2m-1), &&  d=\lambda \binom{n-1}{h-1}, && g=\dfrac{h}{\gcd(n,h)},\\
     &&  \cc H&=\lambda K_m^h, && \cc F= \lambda K_n^h.
\end{align*}
\begin{enumerate}[label=\textup{(\Roman*)}]
    \item A proper $d$-coloring of $\cc H$  can be extended to a 1-factorization of $\cc F$ if and only if $n \equiv 0 \Mod h$.
    \item A partial $g$-factorization of $\cc H$ using $d/g$ colors can be extended to a  $g$-factorization of $\cc F$.
    \item A partial $g$-factorization of $\cc H$ using $k:=d/g$ colors can be extended to a connected $g$-factorization of $\cc F$ if and only if  $\cc H(i)$ is $g$-irregular for $i\in [k]$ and  $n \notequiv 0 \Mod h$.
    \item A partial 2-factorization of $\cc H$ using $k:=\rounddown{d/2}$ colors can be extended to a connected 2-factorization of $\cc F$ if and only if  $\cc H(i)$ is 2-irregular for $i\in [k]$, $2n \equiv 0 \Mod h$, and  $d \equiv 0 \Mod 2$.
    \item A partial $h$-factorization of $\cc H$ using $k:=\rounddown{d/h}$ colors can be extended to a connected $h$-factorization of $\cc F$ if and only if  $h\geq 2$, $\cc H(i)$ is $h$-irregular for $i\in [k]$, and  $d \equiv 0 \Mod h$.
   \end{enumerate}
\end{corollary}
\begin{proof}
To prove  (I), apply Theorem \ref{thm:general} with $k=d, \vb{r}=(1,\dots, 1)$. To prove (II) and (III), let $\vb{r}=(g,\dots,g)$. Using elementary number theory, one can show that $(n,h,\lambda, \vb{r})$ is admissible. Applying Theorem \ref{thm:general} completes the proof of (II). Moreover, $g\geq 2$ if and only if $h\notdivides n$. Applying Theorem \ref{connthmcomb} completes the proof of (III). Proof of (IV) and (V) is obtained by applying Theorem \ref{connthmcomb}  with $\vb{r}=(2,\dots,2)$ and $\vb{r}=(h,\dots,h)$, respectively. 
\end{proof}
\begin{remark}\textup{
(I) is a hypergraph analogue of Cruse's theorem \cite{MR0329925}  that a  proper $(n-1)$-coloring of $K_m$  can be extended to a proper $(n-1)$-coloring of $K_n$ whenever $n$ is even and $n\geq 2m$. Baranyai \cite{MR0416986} constructed $g$-factorizations for $\cc F$ (among other things) and (II) and (III)  strengthen this.
Both (IV) and (V)   are  hypergraph analogues  of Hilton's theorem on extending path decompositions of $K_m$ to Hamiltonian decompositions of $K_n$ \cite{MR746544}. More general results are provided in Section \ref{compappli}. 
}\end{remark}

\section{Proof of Theorem \ref{thm:general}} \label{proof1sec}
Suppose that a partial $\vb r$-factorization of $\lambda K_m^h$ is extended to an $\vb r$-factorization of $\lambda K_n^h$.  In \eqref{divcond+h1}, we established that    $(n,h,\lambda, \vb r)$ must be admissible. 

Conversely, suppose that $(n,h,\lambda, \vb r)$ is admissible, $n \geq (h-1)(2m-1)$, and that a partial $\vb r$-factorization of $\cc G:=\lambda K_m^h$ is given. Let $\cc H$ be the hypergraph whose vertex set is $V(\cc G)\cup \{\alpha\}$ and whose edge multiset is the (colored) edge multiset of $\cc G$ together with further (uncolored) edges, containing (possibly multiple copies of) the new vertex $\alpha$, described as follows.
\begin{align}\label{def2vhyp}
    &&
    \mult_\cc H(X\alpha^{h-i})=\lambda\binom{n-m}{h-i} \qquad \forall X\subseteq V(\cc G), |X|=i, 0\leq i\leq h-1.
    &&
\end{align}
By $\mult_\cc H(X\alpha^j)$ we mean the number of occurrences of an $X\alpha^j$-edge, or $*\alpha^j$-edge for short, in $\cc H$, which is an
edge of the form $X\cup \{\alpha^{j}\}$ (so it contains $j$ copies of $\alpha$). Observe that the edges of $\cc G$ are the $*\alpha^0$-edges. The total number of occurrences of a vertex $v$ in all the edges of $\cc H$ is the {\it degree} of $v$, written  $\dg_\cc H(v)$, and $\cc H(i)$ denotes the current color class $i$ of $\cc H$. 
In the next three subsections, we color the remaining edges of $\cc H$.

\subsection{Coloring the \texorpdfstring{$\bm{*\alpha^i}$}{p}-edges, \texorpdfstring{$\bm{i\in [h-2]}$}{p}} \label{coloringIsubsec}
We color the $*\alpha$-edges, $*\alpha^2$-edges, \dots, $*\alpha^{h-2}$-edges of $\cc H$, in that particular order,  such that 
\begin{align} \label{eq:dg_h-2}
     \dg_{\cc H(j)}(x)\leq r_j
\qquad   \forall x \in V(\cc G), j\in [k].
\end{align}
Suppose to the contrary that for some $i\in [h-2]$, there is an $*\alpha^{i}$-edge $e$ in $\cc H$  that cannot be colored. Let $e=X\cup \{\alpha^i\}$ where $X$ is an $(h-i)$-subset  of $V(\cc G)$. For each $j\in[k]$, there is some $x\in X$ such that $\dg_{\cc H(j)}(x)=r_j$, and consequently, for all $j\in[k]$, $\sum_{x\in X}\dg_{\cc H(j)}(x)\geq r_j$. On the one hand,
$$\sum_{j=1}^k\sum_{x\in X}\dg_{\cc H(j)}(x)\geq \sum_{j=1}^k r_j = \lambda\binom{n-1}{h-1},$$
and on the other hand, by counting the edges colored so far,
$$\sum_{j=1}^k\sum_{x\in X}\dg_{\cc H(j)}(x)\leq \lambda(h-i)\left[\binom{m-1}{h-1} + \sum_{\ell=1}^i\binom{n-m}{\ell}\binom{m-1}{h-\ell-1}-1\right].$$
Thus, we have
\begin{align*}
    \lambda\binom{n-1}{h-1}\leq \lambda(h-i)\left[\binom{m-1}{h-1}+\sum_{\ell=1}^i\binom{n-m}{\ell}\binom{m-1}{h-\ell-1}-1\right].
\end{align*}
We  prove that this is a contradiction by establishing that $f(i)> 0$ for $i\in [h-2]$ where
\begin{align*}
    f(i) &:= \dbinom{n-1}{h-1} - (h-i)\left[\sum_{\ell=0}^{i}\dbinom{n-m}{\ell}\dbinom{m-1}{h-\ell-1} - 1\right], \quad i\in [h-2].
\end{align*}
Since 
\begin{align*}
    f(h-2) =& \binom{n-1}{h-1} - 2\left[\sum_{\ell=0}^{h-2}\binom{n-m}{\ell}\binom{m-1}{h-\ell-1}-1\right]\\
    &= \binom{n-1}{h-1} - 2\left[ \binom{n-1}{h-1} -  \binom{n-m}{h-1}\right]+2\\
    &= 2\binom{n-m}{h-1} - \binom{n-1}{h-1} + 2,
\end{align*}
the following proves that $f(h-2)> 0$.
\begin{align*}
    \binom{n-m}{h-1}\bigg/\binom{n-1}{h-1} &= \dfrac{(n-m)!(n-h)!}{(n-1)!(n-m-h+1)!}\\
    &= \prod_{i=1}^{h-1}\dfrac{n-m-i+1}{n-i}= \prod_{i=1}^{h-1}\left(1-\dfrac{m-1}{n-i}\right)\\
    &\geq \prod_{i=1}^{h-1}\left(1-\dfrac{m-1}{n-h+1}\right)
    = \left(1 - \dfrac{m-1}{n-h+1}\right)^{h-1}\\
   &\geq 1 - \dfrac{(h-1)(m-1)}{n-h+1}\\
    &\geq 1 - \dfrac{(h-1)(m-1)}{(h-1)(2m-1)-(h-1)}
    = \frac{1}{2}.
\end{align*}
To show $f(i)> 0$ for $i\in [h-3]$, let
\begin{align*}
    g(i) &= f(i+1) - f(i)\\
    &= \sum_{\ell=0}^{i+1}\binom{n-m}{\ell}\binom{m-1}{h-\ell-1} - (h-i)\binom{m-1}{h-i-2}\binom{n-m}{i+1} - 1, \quad i\in [h-4]. 
\end{align*}
Since 
\begin{align*}
g(i)&-g(i+1) =  (h-i-2)\binom{m-1}{h-i-3}\binom{n-m}{i+2}-(h-i)\binom{m-1}{h-i-2}\binom{n-m}{i+1}\\
&=\binom{n-m}{i+1}\binom{m-1}{h-i-3}\left(\frac{(h-i-2)(n-m-i-1)}{i+2}-\frac{(h-i)(m-h+i+2)}{h-i-2}\right),
\end{align*}
for $i\in[h-4]$,   $g(i)> g(i+1)$ if and only if
\begin{align}\label{gdecequiv}
(h-i-2)^2(n-m-i-1)> (h-i)(i+2)(m-h+i+2).
\end{align}
Since  $i\leq h-4$, we have $\dfrac{1}{h-i-2} \leq \dfrac{1}{2}$, $\dfrac{h-i}{h-i-2} \leq 2$, and so $\dfrac{h-i}{(h-i-2)^2} \leq 1$. Therefore, 
\begin{align*}
    \dfrac{(h-i)(i+2)(m-h+i+2)}{(h-i-2)^2} + m+i+1 \leq    (h-2)(m-2)+m+h-3< n.
\end{align*}
This proves \eqref{gdecequiv}, and consequently, $g$ is strictly decreasing for $i\in[h-4]$. Thus, there exists an $a$ with $0\leq a\leq h-4$ such that $g(i)\geq 0$ for $1\leq i\leq a$ and $g(i)\leq 0$ for $a+1\leq i\leq h-4$. Therefore, $f(a+1)>f(a)>\dots>f(1)$ and  $f(a+2)>f(a+3)>\dots>f(h-3)$. So, if we show that $f(1)>0$ and $f(h-3)>0$, then we are done.

Since 
\begin{align*}
 f(1) &= \binom{n-1}{h-1} - (h-1)\left[\binom{m-1}{h-1} + (n-m)\binom{m-1}{h-2} - 1\right]\\
    &>  \binom{n-1}{h-1} - (h-1)\binom{m-1}{h-1} - (h-1)(n-m)\binom{m-1}{h-2} \\
    &= \binom{n-1}{h-1} - \binom{m-1}{h-2}\Big((h-1)(n-m)+m-h+1\Big)\\
    &> \binom{n-1}{h-1} - (h-1)(n-1)\binom{m-1}{h-2},
\end{align*}
the following proves that $f(1)> 0$. 
\begin{align*}
    \dfrac{\dbinom{n-1}{h-1}}{\dbinom{m-1}{h-2}}&= \frac{(n-1)!(m-h+1)!}{(h-1) (n-h)!(m-1)!} 
    = \frac{n-1}{h-1}\prod_{i=1}^{h-2}\frac{n-i-1}{m-i}\\
    &= \frac{n-1}{h-1}\prod_{i=1}^{h-2}\left(1 + \frac{n-m-1}{m-i}\right)\\
   & \geq \frac{n-1}{h-1}\prod_{i=1}^{h-2}\left(1 + \frac{n-m-1}{m}\right)\\
    &= \frac{n-1}{h-1}\left(1 + \frac{n-m-1}{m}\right)^{h-2}\\
   & \geq \frac{n-1}{h-1}\left(1 + \frac{(h-2)(n-m-1)}{m}\right) \\
   & \geq \frac{n-1}{h-1}\left(1 + \frac{hm(h-2)}{m}\right)=(n-1)(h-1).
\end{align*}
Now, we show that $f(h-3)>0$. 
Since 
\begin{align*}
    f(h-3) =& \binom{n-1}{h-1} - 3\left[\sum_{\ell=0}^{h-3}\binom{n-m}{\ell}\binom{m-1}{h-\ell-1} - 1\right]\\
    &= \binom{n-1}{h-1} - 3\left[ \binom{n-1}{h-1} -  \binom{n-m}{h-1} - (m-1)\binom{n-m}{h-2}\right]+3\\
    &= 3\binom{n-m}{h-1} + 3(m-1)\binom{n-m}{h-2} - 2\binom{n-1}{h-1} + 3,
\end{align*}
the following proves that $f(h-3)> 0$. 
\begin{align*}
    &\left[\binom{n-m}{h-1}+(m-1)\binom{n-m}{h-2}\right]\bigg/\binom{n-1}{h-1} \\
    &\qquad= \binom{n-m}{h-2}\left(\frac{n-m-h+2}{h-1} + m - 1\right)\bigg/\binom{n-1}{h-1}\\
    &\qquad= \left(\frac{n-m-h+2}{h-1} + m - 1\right)\dfrac{(h-1)(n-m)!(n-h)!}{(n-1)!(n-m-h+2)!}\\
    &\qquad= \left(\frac{n-m-h+2}{h-1} + m - 1\right)\frac{h-1}{n-m-h+2}\prod_{i=1}^{h-1}\dfrac{n-m-i+1}{n-i}\\
    &\qquad\geq \left(1 + \frac{(h-1)(m-1)}{n-h+1}\right) \prod_{i=1}^{h-1}\left(1-\dfrac{m-1}{n-i}\right)\\
    &\qquad\geq \left(1 + \frac{(h-1)(m-1)}{n-h+1}\right) \left(1 - \dfrac{m-1}{n-h+1}\right)^{h-1}\\
   &\qquad\geq \left(1 + \frac{(h-1)(m-1)}{n-h+1}\right) \left(1 - \dfrac{(h-1)(m-1)}{n-h+1}\right)\\
   &\qquad\geq 1 - \left(\frac{(h-1)(m-1)}{2(h-1)(m-1)}\right)^2
   = 1 - \frac{1}{4}  > \frac{2}{3}.
\end{align*}

\subsection{Coloring the \texorpdfstring{$\bm{*\alpha^{h-1}}$}{p}-edges}  We  color  the $*\alpha^{h-1}$-edges  such that
$$
\mult_{\cc H(j)}(x\alpha^{h-1})=r_j - \dg_{\cc H(j)}(x) \quad \forall x\in V(\cc G), j\in [k].
$$
This is possible because for  $x\in V(\cc G)$,
\begin{align*}
    \sum_{j=1}^k\left(r_j-\dg_{\cc H(j)}(x)\right) &= \sum_{j=1}^k r_j - \sum_{j=1}^k\dg_{\cc H(j)}(x)\\
    &= \lambda\binom{n-1}{h-1} - \dg_\cc H(x)\\
    &= \lambda\binom{n-1}{h-1} - \sum_{\ell=1}^{h-1}\lambda\binom{m}{\ell}\binom{n-m}{h-\ell-1}\\
    &= \lambda\sum_{\ell=0}^{h-1}\binom{m}{\ell}\binom{n-m}{h-\ell-1} - \lambda\sum_{\ell=1}^{h-1}\binom{m}{\ell}\binom{n-m}{h-\ell-1}\\
    &= \lambda\binom{n-m}{h-1}.
\end{align*}
\subsection{Coloring the \texorpdfstring{$\bm{\alpha^{h}}$}{p}-edges} 
Recall that $n\geq h m$ and   $h\divides r_j n$ for $j\in [k]$. Hence, we color the $\alpha^{h}$-edges such that
\begin{align*}
    \mult_{\cc H(j)}(\alpha^h) = \frac{r_j n}{h} - r_j m + \sum_{\ell=0}^{h-2}(h-\ell-1)\mult_{\cc H(j)}(*\alpha^{\ell}) \quad \forall j\in[k].
\end{align*}
Let us verify that this is in fact possible. 
 \begin{align*}
     \sum_{j=1}^k\mult_{\cc H(j)}(\alpha^h) =& \sum_{j=1}^k\left(\frac{r_j n}{h} - r_j m + \sum_{\ell=0}^{h-2}(h-\ell-1)\mult_{\cc H(j)}( \alpha^{\ell})\right)\\
     &= \frac{n}{h}\sum_{j=1}^k r_j - m\sum_{j=1}^k r_j + \sum_{j=1}^k\sum_{\ell=0}^{h-2}(h-\ell-1)\mult_{\cc H(j)}(*\alpha^{\ell})\\
     &= \lambda\frac{n}{h}\binom{n-1}{h-1} - \lambda m\binom{n-1}{h-1} + \sum_{\ell=0}^{h-2}(h-\ell-1)\sum_{j=1}^k\mult_{\cc H(j)}(*\alpha^{\ell})\\
     &= \lambda\binom{n}{h} - \lambda m\binom{n-1}{h-1} + \sum_{\ell=0}^{h-2}\lambda(h-\ell-1)\binom{m}{h-\ell}\binom{n-m}{\ell}\\
     &= \lambda\sum_{\ell=0}^h\binom{m}{h-\ell}\binom{n-m}{\ell} - \lambda m\binom{n-1}{h-1} \\ 
     &\quad +\sum_{\ell=0}^{h-2}\lambda(h-\ell)\binom{m}{h-\ell}\binom{n-m}{\ell}- \sum_{\ell=0}^{h-2}\lambda\binom{m}{h-\ell}\binom{n-m}{\ell}\\
     &= \lambda m\binom{n-m}{h-1} + \lambda\binom{n-m}{h} - \lambda m\binom{n-1}{h-1} \\ 
     &\qquad\qquad\qquad\quad +\sum_{\ell=0}^{h-2}\lambda(h-\ell)\binom{m}{h-\ell}\binom{n-m}{\ell}\\
     &= \lambda\binom{n-m}{h} - \lambda m\binom{n-1}{h-1} + \sum_{\ell=0}^{h-1}\lambda(h-\ell)\binom{m}{h-\ell}\binom{n-m}{\ell}\\
     &= \lambda\binom{n-m}{h} - \lambda m\binom{n-1}{h-1} + \lambda h\binom{m}{h} + \lambda m\left[\binom{n-1}{h-1} - \binom{m-1}{h-1}\right]\\
     &= \lambda\binom{n-m}{h} + \lambda m\binom{n-1}{h-1} - \lambda m\binom{n-1}{h-1} + \lambda h\binom{m}{h} - \lambda m\binom{m-1}{h-1}\\
     &= \lambda\binom{n-m}{h}.
 \end{align*}
\subsection{Regularity of the Coloring of  \texorpdfstring{$\bm{\cc H}$}{p}} \label{coloringIVsubsec}
As a result of the coloring of the $*\alpha^{h-1}$-edges, we have
\begin{equation} \label{degreepart1}
    \dg_{\cc H(j)}(x) = r_j \quad \forall x\in V(\cc G), j\in [k],
\end{equation}
and so
\begin{align} \label{rjmsum}
r_j m = \sum_{x\in V(\cc G)} \dg_{\cc H(j)}(x) = \sum_{\ell=0}^{h-1}(h-\ell)\mult_{\cc H(j)}(*\alpha^{\ell})\quad  \forall j\in [k].    
\end{align}
Hence,
\begin{align*}
    \dg_{\cc H(j)}(\alpha) &= \sum_{\ell=0}^h \ell\mult_j(*\alpha^{\ell})\\
    &= h\mult_j(*\alpha^h) + \sum_{\ell=0}^{h-1}\ell\mult_j(*\alpha^{\ell})\\
    &= h\left(\frac{r_jn}{h} - r_jm + \sum_{\ell=0}^{h-2}(h-\ell-1)\mult_j(*\alpha^{\ell})\right) + \sum_{\ell=0}^{h-1}\ell\mult_j(*\alpha^{\ell})\\
    &= r_jn - hr_jm + \sum_{\ell=0}^{h-2}h(h-\ell-1)\mult_j(*\alpha^{\ell}) + \sum_{\ell=0}^{h-1}\ell\mult_j(*\alpha^{\ell})\\
    &= r_jn - h\sum_{\ell=0}^{h-1}(h-\ell)\mult_j(*\alpha^{\ell}) + \sum_{\ell=0}^{h-2}h(h-\ell-1)\mult_j(*\alpha^{\ell}) + \sum_{\ell=0}^{h-1}\ell\mult_j(*\alpha^{\ell})\\
    &= r_jn - h\mult_j(*\alpha^{h-1}) - h\sum_{\ell=0}^{h-2}(h-\ell)\mult_j(*\alpha^{\ell}) + h\sum_{\ell=0}^{h-2}(h-\ell-1)\mult_j(*\alpha^{\ell})\\
    &\qquad+ \sum_{\ell=0}^{h-1}\ell\mult_j(*\alpha^{\ell})\\
    &= r_jn - h\mult_j(*\alpha^{h-1}) - h\sum_{\ell=0}^{h-2}\mult_j(*\alpha^{\ell}) + \sum_{\ell=0}^{h-1}\ell\mult_j(*\alpha^{\ell})\\
    &= r_jn - h\sum_{\ell=0}^{h-1}\mult_j(*\alpha^{\ell}) + \sum_{\ell=0}^{h-1}\ell\mult_j(*\alpha^{\ell})\\
    &= r_jn - \sum_{\ell=0}^{h-1}(h-\ell)\mult_j(*\alpha^{\ell})\\
    &= r_j n - r_j m = r_j(n-m) \qquad \forall j\in [k].
\end{align*}
\subsection{A Fair Detachment of  \texorpdfstring{$\bm{\cc H}$}{p}} \label{fairetachmsubsec}
By \cite[Theorem 4.1]{MR2942724}, there exists an $n$-vertex hypergraph $\cc F$, called the {\it fair $(\alpha, n-m)$-detachment} of $\cc H$, obtained by replacing the vertex $\alpha$ of $\cc H$ by $n-m$ new vertices  $\alpha_1,\dots,\alpha_{n-m}$ in $\cc F$ and replacing each $X\alpha^i$-edge by an edge of the form $X\cup U$ where $U\subseteq \{\alpha_1,\dots,\alpha_{n-m}\}, |U|=i\in [h]$ (leave the remaining vertices and edges of $\cc H$ intact),   such that the edges incident with $\alpha$ (in each color class of $\cc H$) are shared as evenly as possible among $\alpha_1,\dots,\alpha_{n-m}$ in $\cc F$ in the following way.
\begin{itemize}
    \item [(a)] For  $i\in [n-m],j\in[k]$, 
    \begin{align*} 
        \dg_{\cc F(j)}(\alpha_i)=\dfrac{\dg_{\cc H(j)}(\alpha)}{n-m}=\dfrac{r_j(n-m)}{n-m}=r_j;
    \end{align*}
    \item [(b)] For $X\subseteq V(\cc G), U\subseteq \{\alpha_1,\dots,\alpha_{n-m}\}, |X|=h-i, |U|=i\in [h]$, 
    \begin{align} \label{multcond}
    \mult_{\cc F}(X\cup U)=\dfrac{\mult_{\cc H}(X\alpha^i)}{\dbinom{n-m}{i}}=\dfrac{\lambda\dbinom{n-m}{i}}{\dbinom{n-m}{i}}=\lambda.
    \end{align}
\end{itemize}
Observe that by (b), $\cc F \cong \lambda K_n^h$, and by \eqref{degreepart1} and (a), $\cc F(i)$ is an $r_i$-factor for $i\in [k]$. This completes the proof of Theorem \ref{thm:general}.

\section{Proof of Theorem \ref{incompembthm}}

Suppose that $n \geq (h-1)(2m-1)$ and $(n,h,\lambda, \vb r)$ is admissible where $\vb{s}:=(s_1,\dots,s_q), \vb{r}:=(r_1,\dots,r_k)$ such that $$\sum_{i=1}^q \rounddown{\frac{(r_i-s_i)m}{h}}+\sum_{i=q+1}^k \rounddown{\frac{r_im}{h}} \geq \lambda\binom{m}{h}.$$
 It suffices to extend the given partial $\vb{s}$-factorization of $\mathcal G\subseteq \cc G_1:=\lambda K_m^h$ to a partial $\vb r$-factorization of $\cc G_1$, for,  by Theorem \ref{thm:general}, we may extend this partial $\vb r$-factorization of $\cc G_1$ to an $\vb r$-factorization of $\lambda K_n^h$.

Let $\cc H$ be a hypergraph whose vertex set is $\{\alpha\}$ and has ${\lambda \binom{m}{h}}$ copies of an edge of the form $\{\alpha^h\}$. In other words,
$$
V(\cc H)=\{\alpha\}, \quad \mult_\cc H(\alpha^h)=\lambda\binom{m}{h}.
$$
We color the edges of $\cc H$ such that
\begin{align*}
  \mult_{\cc H(i)}(\alpha^h)\leq
  \begin{cases}
    \rounddown{\dfrac{(r_i-s_i) m}{h}}
    &
    {\text{for}}\ i\in [q],
    \\[10pt]
    \rounddown{\dfrac{r_i m}{h}}
    &
    \text{for}\ i\in [k]\backslash[q].
  \end{cases}
\end{align*}
By \cite[Theorem 4.1]{MR2942724}, there exists an $m$-vertex hypergraph $\cc F$ obtained by replacing the vertex $\alpha$ of $\cc H$ by $m$ new vertices $\alpha_1,\dots,\alpha_m$ in $\cc F$ and replacing each $\alpha^h$-edge by an edge of the form $U$ where $U\subseteq \{\alpha_1,\dots,\alpha_{m}\}, |U|=h$, such that the edges incident with $\alpha$ (in each color class of $\cc H$) are shared as evenly as possible among $\alpha_1,\dots,\alpha_m$ in $\cc F$ in the following way. 
\begin{align} 
        \dg_{\cc F(j)}(\alpha_i)&\approx\dfrac{\dg_{\cc H(j)}(\alpha)}{m}\leq\dfrac{h}{m}\rounddown{\dfrac{(r_j-s_j) m}{h}}\leq r_j-s_j  &\forall i\in [m],j\in [q]; \label{degcondincompnew}\\[10pt]
        \dg_{\cc F(j)}(\alpha_i)&\approx\dfrac{\dg_{\cc H(j)}(\alpha)}{m}\leq\dfrac{h}{m}\rounddown{\dfrac{r_j m}{h}}\leq r_j  &\forall i\in [m],j\in [k]\backslash[q];\label{degcondincomp} \\[10pt]
        \mult_{\cc F}(U)&=\mult_{\cc H}(\alpha^h)\Big/\binom{m}{h}=\lambda\dbinom{m}{h}\bigg/\dbinom{m}{h}=\lambda  &\forall U\subseteq \{\alpha_1,\dots,\alpha_{n}\}, |U|=h.  \label{multcondincomp}      
\end{align}
Here, $x \approx y$ means $\lfloor y \rfloor\leq x\leq  \lceil y \rceil$. By \eqref{multcondincomp}, $\cc F \cong \lambda K_m^h$, and by \eqref{degcondincompnew} and \eqref{degcondincomp}, the coloring of $\cc F$ induces a partial $(r_1-s_1,\dots,r_q-s_q,r_{q+1},\dots,r_k)$-factorization. We color each edge of $\cc G_1 \backslash \mathcal G$ with the color of the corresponding edge in $\cc F$. This leads to a partial $\vb r$-factorization of $\cc G_1$, as desired, and completes the proof of Theorem \ref{incompembthm}.

\section{Proof of Theorem \ref{connthmcomb}}
 A vertex $\alpha$ in a connected hypergraph $\mathcal G$ is a {\it cut vertex} if there exist two non-trivial sub-hypergraphs $I, J$ of $\mathcal G$ such that  $I\cup J=\mathcal G$,  $V( I\cap  J)=\{\alpha\}$, and  $E( I\cap  J)=\varnothing$. A sub-hypergraph $W$ of a  hypergraph $\mathcal G$ is an {\it $\alpha$-wing} of $\mathcal G$ if (i) $W$ is non-trivial and connected, (ii) $\alpha$ is not a cut vertex of $W$, and (iii) no edge in $E(\mathcal G)\backslash E(W)$ is incident with a vertex in $V(W)\backslash \{\alpha\}$. An $\alpha$-wing $W$ is  {\it large} if $V(W)\neq \{\alpha\}$ and is  {\it small} if $V(W)= \{\alpha\}$. Let $\omega_{\alpha}(\cc G)$ and $\omega^L_{\alpha}(\cc G)$ be the number of $\alpha$-wings and  the number of large $\alpha$-wings in $\cc G$, respectively. Let $c(\cc G)$ denote the number of components of $\cc G$. 

An $r$-factor cannot be connected unless $r\geq 2$.  Moreover, if a component of a color class of $\lambda K_m^h$  is $r$-regular, then there is no way to extend it to a connected $r$-factor in $\lambda K_n^h$. This justifies the necessity of the conditions in Theorem \ref{connthmcomb}.

Now, suppose that $(n,h,\lambda, \vb r)$ is admissible, $n \geq (h-1)(2m-1)$, and  a partial $\vb r$-factorization of $\cc G:=\lambda K_m^h$ is given. Let $\cc H$ be the hypergraph defined in Section \ref{proof1sec} whose edges are colored according to the coloring described in Subsections \ref{coloringIsubsec}--\ref{coloringIVsubsec}.

Let us fix $j\in [k]$ such that $r_j \geq 2$ and no component of $\cc G(j)$ is $r_j$-regular.  Since $\dg_{\cc H(j)}(u) = r_j$ for all $u\in V(\cc G)$, there is at least one edge joining $\alpha$ and each component of $\cc G(j)$. Hence, $\cc H(j)$ must be connected. We  prove that $\cc F(j)$ --- constructed in Subsection \ref{fairetachmsubsec} --- is connected. 
\subsection{An Upper Bound for the Number of Wings} We claim that
\begin{equation}  \label{wingcondcor1cam}
\omega_\alpha(\cc H(j))\leq (r_j-1)(n-m)+1.
\end{equation}
As each component of $G(j)$ is itself connected, it cannot correspond to more than one large $\alpha$-wing in $H(j)$. Each of these components may correspond to an $\alpha$-wing in $H(j)$ if there is some $*\alpha^i$-edge in $H(j)$ containing some vertex in the component. Since every $\alpha^h$-edge is a small $\alpha$-wing in $\cc H(j)$ and each component of $\cc G(j)$ corresponds to at most one large  $\alpha$-wing in $\cc H(j)$, we have 
\begin{align*}
\omega_\alpha(\cc H(j))&=\mult_{\cc H(j)}(\alpha^h)+\omega^L_\alpha(\cc H(j))\\
&\leq  \mult_{\cc H(j)}(\alpha^h)+c(\cc G(j))\\
& = \frac{r_jn}{h} - r_jm + \sum_{i=1}^{h-1}(h-i)\mult_{\cc H(j)}(*\alpha^{i-1}) + c(\cc G(j)).
\end{align*}
Thus, to prove \eqref{wingcondcor1cam}, it suffices to show that
\begin{align*}
    \frac{r_jn}{h} - r_jm + \sum_{i=1}^{h-1}(h-i)\mult_{\cc H(j)}(*\alpha^{i-1}) + c(\cc G(j)) &\leq (r_j-1)(n-m) + 1,
\end{align*}
which is equivalent to showing
\begin{equation} \label{wingcondv2}
    r_j n\left(1-\frac{1}{h}\right) - n + m + 1 \geq \sum_{i=1}^{h-1}(h-i)\mult_{\cc H(j)}(*\alpha^{i-1}) + c(\cc G(j)).
\end{equation}

For $i\in[h-1]$, an $*\alpha^{h-i}$-edge in $\cc H$ contains $i$ vertices of $V(\cc G)$ which are contained in at most $i$ different components of $\cc G(j)$. Thus, an $*\alpha^{h-i}$-edge connects at most $i$ components of $\cc G(j)$. Since $\cc H(j)$ is connected, we have
\begin{equation*}
c(\cc G(j)) \leq \sum_{i=1}^{h-1}i\mult_{\cc H(j)}(* \alpha^{h-i}). 
\end{equation*}
Therefore,
\begin{equation} \label{components}
\frac{c(\cc G(j)}{h-1} \leq \sum_{i=1}^{h-1}\frac{i}{h-1}\mult_{\cc H(j)}(*\alpha^{h-i}) \leq \sum_{i=1}^{h-1}(h-i)\mult_{\cc H(j)}(*\alpha^{h-i}).
\end{equation}
Using \eqref{rjmsum}, we have
\begin{align} \label{rmlong}
    r_jm \left ( 1- \frac{1}{h}\right)& - \frac{1}{h}\sum_{i=1}^{h-1}(h-i)\mult_{\cc H(j)}(*\alpha^{h-i}) \nonumber\\
    =&  \left ( 1- \frac{1}{h}\right)\sum_{i=1}^{h}i\mult_{\cc H(j)}(*\alpha^{h-i}) - \frac{1}{h}\sum_{i=1}^{h-1}(h-i)\mult_{\cc H(j)}(*\alpha^{h-i}) \nonumber\\
    =& \sum_{i=1}^{h-1} \left[i\left ( 1- \frac{1}{h}\right)-\left(1-\frac{i}{h}\right)\right]\mult_{\cc H(j)}(*\alpha^{h-i}) + (h-1)\mult_{\cc H(j)}(*\alpha^{0})\nonumber\\
    =& \sum_{i=1}^{h-1} (i-1)\mult_{\cc H(j)}(*\alpha^{h-i}) + (h-1)\mult_{\cc H(j)}(*\alpha^{0})\nonumber\\
    =& \sum_{i=1}^h (i-1)\mult_{\cc H(j)}(*\alpha^{h-i})\nonumber\\
    =& \sum_{i=1}^{h-1} (h-i)\mult_{\cc H(j)}(*\alpha^{i-1}).
    \end{align}
Moreover, $c(\cc G(j))\leq m$. Now as $n\geq hm$ and  $r_j\geq 2$, we have the following which proves \eqref{wingcondv2}. 
\begin{align*}
    r_jn \left ( 1- \frac{1}{h}\right) - n + m + 1 &\geq hm\left(r_j \left ( 1- \frac{1}{h}\right) - 1\right) + m + 1\\
    &= r_jm \left ( 1- \frac{1}{h}\right) + m\Big ((h-1)r_j - h\Big)\left ( 1- \frac{1}{h}\right) + 1\\
    &\geq r_jm \left ( 1- \frac{1}{h}\right) +m(h-2)\left ( 1- \frac{1}{h}\right) + 1\\
    &\geq r_jm \left ( 1- \frac{1}{h}\right) + \left ( 1- \frac{1}{h(h-1)}\right)m\\
    &\geq r_jm \left ( 1- \frac{1}{h}\right) + \left ( 1- \frac{1}{h(h-1)}\right)c(\cc G(j))\\
    &= r_jm \left ( 1- \frac{1}{h}\right)  - \frac{1}{h}\left(\frac{c(\cc G(j))}{h-1}\right)+c(\cc G(j))\\
    &\geq r_jm \left ( 1- \frac{1}{h}\right) - \frac{1}{h}\sum_{i=1}^{h-1}(h-i)\mult_{\cc H(j)}(*\alpha^{h-i}) + c(\cc G(j))\\
    &= \sum_{i=1}^{h-1}(h-i)\mult_{\cc H(j)}(*\alpha^{i-1}) + c(\cc G(j)).
\end{align*}

\subsection{Connected Detachments} By \cite[Theorem 1.1]{2020arXiv200909674B}, in the fair $(\alpha, n-m)$-detachment $\cc F$ of $\cc H$,  $\cc F(j)$ is connected if and only if
\begin{align}
    \dg_{\cc H(j)}(\alpha)-\omega_\alpha(\cc H(j))\geq n-m-1.
\end{align}
In \eqref{wingcondcor1cam}, we showed that
\begin{equation*}
\omega_\alpha(\cc H(j))\leq (r_j-1)(n-m)+1.
\end{equation*}
Moreover, recall that $\dg_{\cc H(j)}(\alpha)=r_j (n-m)$. Hence, 
\begin{align*}
    \dg_{\cc H(j)}(\alpha)-\omega_\alpha(\cc H(j))\geq r_j (n-m)-(r_j-1)(n-m)-1= n-m-1.
\end{align*}
This completes the proof of Theorem \ref{connthmcomb}.
\section{Proof of Theorem \ref{connthmcomb2}}

Suppose that $n \geq (h-1)(2m-1)$ and $(n,h,\lambda, \vb r)$ is admissible where $\vb{s}=(s_1,\dots,s_q), \vb{r}=(r_1,\dots,r_k)$ such that
$$\sum_{i\in B} \rounddown{\frac{(\overline {r}_i-s_i)m}{h}}+\sum_{i\in [k]\backslash [q]} \rounddown{\frac{\overline {r}_i m}{h}} \geq \lambda\binom{m}{h},$$
where $A\subseteq \{i \in [k]\ |\ r_i\geq 2\}, B=\{i\in [q] \ |\ r_i\neq s_i\}$,   $\overline {r}_i:=r_i-1$ if $i\in A$, and  $\overline {r}_i:=r_i$ if $i\in [k]\backslash A$. 
Assume that a partial $\vb{s}$-factorization of $\mathcal G\subseteq \cc G_1:=\lambda K_m^h$ is given such that $\cc G(i)$ is $r_i$-irregular for $i\in A$.
By Theorem \ref{connthmcomb}, it suffices to extend the given partial $\vb{s}$-factorization of $\mathcal G$ to a partial $\vb r$-factorization of $\cc G_1$ in such a way that no component of color class $i$ of the partial $\vb r$-factorization of $\cc G_1$ is $r_i$-regular for $i \in A$.

Let $\cc H$ be a hypergraph whose vertex set is $\{\alpha\}$ and which has ${\lambda \binom{m}{h}}$ copies of an edge of the form $\{\alpha^h\}$. We color the edges of $\cc H$ such that
\begin{align*}
  \mult_{\cc H(i)}(\alpha^h)\leq
  \begin{cases}
    \rounddown{\dfrac{(r_i-s_i) m}{h}}
    &
    {\text{for}}\ i\in [q]\backslash (A\cap B),
    \\[10pt]
    \rounddown{\dfrac{(r_i-s_i-1) m}{h}}
    &
    {\text{for}}\ i\in A \cap B,
    \\[10pt]
    \rounddown{\dfrac{r_i m}{h}}
    &
    {\text{for}}\ i\in ([k]\backslash[q])\backslash A,
    \\[10pt]
    \rounddown{\dfrac{(r_i-1) m}{h}}
    &
    {\text{for}}\  i\in A\cap ([k]\backslash[q]).
  \end{cases}
\end{align*}
By \cite[Theorem 4.1]{MR2942724}, there exists an $m$-vertex hypergraph $\cc F$ obtained by replacing the vertex $\alpha$ of $\cc H$ by $m$ new vertices  $\alpha_1,\dots,\alpha_m$ in $\cc F$ and replacing each $\alpha^h$-edge by an edge of the form $U$ where $U\subseteq \{\alpha_1,\dots,\alpha_{m}\}, |U|=h$, such that the edges incident with $\alpha$ (in each color class of $\cc H$) are shared as evenly as possible among $\alpha_1,\dots,\alpha_m$ in $\cc F$ in the following way. 
\begin{align*}
        \mult_{\cc F}(U)&=\mult_{\cc H}(\alpha^h)\Big/\binom{m}{h}=\lambda\dbinom{m}{h}\bigg/\dbinom{m}{h}=\lambda  &\forall U\subseteq \{\alpha_1,\dots,\alpha_{n}\}, |U|=h;
\end{align*}
\begin{align*}
 \forall i\in [m]: \dg_{\cc F(j)}(\alpha_i)&\approx\dfrac{\dg_{\cc H(j)}(\alpha)}{m}\leq \dfrac{h \mult_{\cc H(j)}(\alpha^h)}{m} \\
  &\leq \begin{cases}
    r_j-s_j
    &
    {\text{for}}\ j\in [q]\backslash (A\cap B),
    \\
    r_j-s_j-1
    &
    {\text{for}}\ j\in A \cap B,
    \\
    r_j
    &
    {\text{for}}\ j\in ([k]\backslash[q])\backslash A,
    \\
    r_j-1
    &
    {\text{for}}\  j\in A\cap ([k]\backslash[q]).
  \end{cases}
\end{align*}

To obtain a partial  $\vb r$-factorization of $\cc G_1$, we color each edge of $\cc G_1 \backslash \mathcal G$ with the color of the corresponding edge in $\cc F$. Recall that no component of color class $i$ of the  partial $\vb s$-factorization of $\cc G$ is $r_i$-regular for $i\in A$. Hence, $\cc G_1(i)$ is $r_i$-irregular for $i\in A$, and we are done.

\section{More Applications} \label{compappli}
Throughout this section,  
\begin{align*}
    && &n \geq (h-1)(2m-1), && \cc G\subseteq \lambda K_m^h, && \cc F:= \lambda K_n^h,\\
    &&  &c:=\lambda \binom{m-1}{h-1}, && d:=\lambda \binom{n-1}{h-1},  &&g:=\dfrac{h}{\gcd(n,h)}, \\
    && &\vb{s}:=(s_1,\dots,s_q), && \vb{r}:=(r_1,\dots,r_k), && 1\leq s_i\leq r_i \mbox { for } i\in [q],\\
    && &A\subseteq \{i \in [k]\ |\ r_i\geq 2\}, && B:=\{i\in [q] \ |\ r_i\neq s_i\},\\
    && &\overline {r}_i:=r_i-1 \mbox { for }i\in A, &&\overline {r}_i:=r_i  \mbox { for } i\in [k]\backslash A.
\end{align*} 
For $i\in [k]$, let $\delta_i, \overline{\delta}_i$ be the integers with $0\leq \delta_i, \overline{\delta}_i\leq h-1$ such that 
\begin{align*}
 \delta_i \equiv \begin{cases}
   (r_i-s_i)m \Mod h
    &
    {\text{for}}\ i\in [q],
    \\
    r_i m \quad \qquad \Mod h
    &
    {\text{for}}\  i\in [k] \backslash [q],
  \end{cases} \qquad 
\overline{\delta}_i \equiv \begin{cases}
   (\overline{r}_i-s_i)m \Mod h
    &
    {\text{for}}\ i\in [q],
    \\
    \overline{r}_i m \quad \qquad \Mod h
    &
    {\text{for}}\  i\in [k] \backslash [q].
  \end{cases}\end{align*}
If $A=[k], \vb{r}=(r,\dots,r), \vb{s}=(s,\dots,s)$, then  $\Delta_1:=\delta_i$, $\overline {\Delta}_1:=\overline {\delta}_i$ for $i\in [q]$, and $\Delta_2:=\delta_i$, $\overline {\Delta}_2:=\overline {\delta}_i$ for $i\in [k] \backslash [q]$. 

\begin{theorem} \label{restatmenttwoconthm}
\begin{enumerate}
    \item [\textup{(a)}] If 
\begin{align*}
d-c\geq \sum_{i\in [q]} s_i +\frac{1}{m}\sum_{i\in [k]} \delta_i,
\end{align*}
then a partial $\vb{s}$-factorization of $\cc G$ can be embedded into an $\vb{r}$-factorization of $\cc F$ if and only if $(n,h,\lambda, \vb r)$ is admissible. 
\item [\textup{(b)}]  If
\begin{align*} 
    \sum_{i\in B\cup ([k] \backslash [q])} \overline{r}_i -c\geq \sum_{i\in B} s_i +\frac{1}{m}\sum_{i\in B\cup ([k] \backslash [q])} \overline{\delta}_i,
\end{align*}
then for  $i\in A$, $\cc F(i)$ is connected if and only if $\cc G(i)$ is $r_i$-irregular.
\end{enumerate}
\end{theorem}
\begin{proof}
Since
\begin{align*}
    \frac{h}{m}\left(\sum_{i\in [q]} \rounddown{\frac{(r_i-s_i)m}{h}}+\sum_{i\in [k] \backslash [q]} \rounddown{\frac{r_i m}{h}}-\frac{m c}{h} \right)&=\sum_{i\in [k]} r_i -\sum_{i\in [q]} s_i -\frac{1}{m}\sum_{i\in [k]} \delta_i-c  \\
    &= d-c-\sum_{i\in [q]} s_i -\frac{1}{m}\sum_{i\in [k]} \delta_i,
\end{align*}
and
\begin{align*}
    \frac{h}{m}\left(\sum_{i\in B} \rounddown{\frac{(\overline{r}_i-s_i)m}{h}}+\sum_{i\in [k] \backslash [q]} \rounddown{\frac{\overline{r}_im}{h}} -\frac{m c}{h} \right)&=\sum_{i\in B\cup ([k] \backslash [q])} \overline{r}_i -\sum_{i\in B} s_i -\frac{1}{m}\sum_{i\in B\cup ([k] \backslash [q])} \overline{\delta}_i-c, 
\end{align*}
the result follows from Theorems \ref{incompembthm} and \ref{connthmcomb2}.
\end{proof}

In particular, we have the following results.
\begin{theorem} \label{restatmenttwoconthmakappabkapp}
If  $s_i<r_i$ for $i\in [q]$ and 
\begin{align*}
d-c-k\geq \sum_{i\in [q]} s_i +\frac{1}{m}\sum_{i\in [k]} \overline{\delta}_i,
\end{align*}
then a partial $\vb{s}$-factorization of $\cc G$ can be embedded into a connected $\vb{r}$-factorization of $\cc F$  if and only if the following hold.
\begin{enumerate}
    \item [\textup{(i)}] $(n,h,\lambda, \vb r)$ is admissible; 
    \item [\textup{(ii)}] $r_i\geq 2$ for $i\in [k] \backslash [q]$. 
\end{enumerate}
\end{theorem}
\begin{proof}
Since $s_i<r_i$ for $i\in [q]$, we have $B=[q]$, $r_i\geq 2$ for $i\in [q]$, and that $\cc G(i)$ is $r_i$-irregular  for $i\in [q]$. Now let $A=[k]$. Since 
\begin{align*} 
    \sum_{i\in [k]} \overline{r}_i -c- \sum_{i\in [q]} s_i -\frac{1}{m}\sum_{i\in [k]} \overline{\delta}_i= d-k -c- \sum_{i\in [q]} s_i -\frac{1}{m}\sum_{i\in [k]} \overline{\delta}_i,
\end{align*}
applying the second part of Theorem \ref{restatmenttwoconthm} completes the proof.
\end{proof}

\begin{theorem} \label{restatmenttwoconthmakappabemp}
If 
\begin{align*}
\sum_{i\in [k] \backslash [q]} r_i-c-k+q\geq \frac{1}{m}\sum_{i\in [k] \backslash [q]} \overline{\delta}_i,
\end{align*}
then a partial $\vb{r}$-factorization of $\cc G$ can be embedded into a connected $\vb{r}$-factorization of $\cc F$  if and only if the following hold.
\begin{enumerate}
    \item [\textup{(i)}] $r_i\geq 2$ for $i\in [k]$;
    \item [\textup{(ii)}]  $(n,h,\lambda, \vb r)$ is admissible;
    \item [\textup{(iii)}] $\cc G(i)$ is $r_i$-irregular for $i\in [k]$.
\end{enumerate}
\end{theorem}
\begin{proof}
Since
\begin{align*}
\sum_{i\in [k] \backslash [q]} \overline{r}_i-c- \frac{1}{m}\sum_{i\in [k] \backslash [q]} \overline{\delta}_i=\sum_{i\in [k] \backslash [q]} r_i-k+q-c- \frac{1}{m}\sum_{i\in [k] \backslash [q]} \overline{\delta}_i,
\end{align*}
applying  the second part of Theorem \ref{restatmenttwoconthm} with $A:=[k], B:=\varnothing$ completes the proof.
\end{proof}

\begin{theorem} \label{rssimpleineqthm}
\begin{enumerate}
    \item [\textup{(a)}] If $$q\leq \dfrac{drm-crm- d \Delta_2}{r(sm+\Delta_1-\Delta_2)},$$ then  a partial $s$-factorization of $\cc G$ can be embedded into an  $r$-factorization of $\cc F$  if and only if $d\equiv 0 \Mod r$ and $rn\equiv 0 \Mod h$. 
    \item [\textup{(b)}] If $s<r$ and
$$q\leq \frac{drm-crm-d\overline{\Delta}_2-dm}{r(sm+\overline{\Delta}_1-\overline{\Delta}_2)},$$ 
then  each  $r$-factor in  $\cc F$ is connected. 
\end{enumerate}
\end{theorem}
\begin{proof} Since $s\geq 1, m\geq h \geq \Delta_2+1$, we have $sm+\Delta_1-\Delta_2> 0$. Since $d=kr$, we have
\begin{align*}
rm \left(d-c- \sum_{i\in [q]} s_i -\frac{1}{m}\sum_{i\in [k]} \delta_i\right)&=rm\left( d-c-qs-\frac{q\Delta_1}{m}-\frac{(k-q)\Delta_2}{m} \right)\\
&=drm-crm-d\Delta_2 -qr(sm+\Delta_1-\Delta_2).
\end{align*}
Applying Theorem \ref{restatmenttwoconthm} with $\vb{s}=(s,\dots,s),\vb{r}=(r,\dots,r)$ proves (a). 

Now, assume that $s<r$. We have $sm+\overline{\Delta}_1-\overline{\Delta}_2> 0$ and 
\begin{align*}
rm\left(d-c-k- \sum_{i\in [q]} s_i -\frac{1}{m}\sum_{i\in [k]} \overline{\delta}_i\right)&=rm\left(d-c-\frac{d}{r}- qs -\frac{q\overline{\Delta}_1}{m}-\frac{(k-q)\overline{\Delta}_2}{m}\right)\\
&=drm-crm-dm-d\overline{\Delta}_2-qr(sm+\overline{\Delta}_1-\overline{\Delta}_2).
\end{align*}
Applying Theorem \ref{restatmenttwoconthmakappabkapp} proves (b). 
\end{proof}
\begin{remark}
If $\Delta_1=\Delta_2=0$,  then  the first bound for $q$ in Theorem \ref{rssimpleineqthm} is $q\leq \frac{d-c}{s}$. Similarly, if $\overline{\Delta}_1=\overline{\Delta}_2=0$,  the second bound for $q$ will be $q\leq \frac{d-c}{s}-\frac{d}{rs}$.
\end{remark}

\begin{theorem} \label{last2boundthm}
\begin{enumerate}
    \item [\textup{(a)}] If $$q\leq \dfrac{d}{r}-\frac{cm}{rm-\Delta_2},$$  then  a partial $r$-factorization of $\cc G$ can be embedded into an  $r$-factorization of $\cc F$  if and only if $d\equiv 0 \Mod r$ and $rn\equiv 0 \Mod h$.
    \item [\textup{(b)}] If 
$$q\leq \frac{d}{r}-\frac{cm}{rm-m-\overline{\Delta}_2},$$ 
then  each  $r$-factor in  $\cc F$ is connected  if and only if $r\geq 2$ and each  color class in $\cc G$ are $r$-irregular.
\end{enumerate}
\end{theorem}
\begin{proof} Using Theorem \ref{rssimpleineqthm} with $r=s$ (and so $\Delta_1=0$), we have
$$\dfrac{drm-crm- d \Delta_2}{r(sm+\Delta_1-\Delta_2)}=\dfrac{d(rm-  \Delta_2)-crm}{r(rm-\Delta_2)}=\frac{d}{r}-\frac{cm}{rm-\Delta_2},$$
which proves (a). 
Applying Theorem \ref{restatmenttwoconthmakappabemp} with $\vb{r}=(r,\dots,r)$, we have
\begin{align*}
rm\left(\sum_{i\in [k] \backslash [q]} r_i-c-k+q- \frac{1}{m}\sum_{i\in [k] \backslash [q]} \overline{\delta}_i\right)&=rm\left((k-q)r-c-k+q-\frac{(k-q)\overline{\Delta}_2}{m}\right)\\
&=rm\left(d-qr-c-\frac{d}{r}+q-\frac{d\overline{\Delta}_2}{rm}+\frac{q \overline{\Delta}_2}{m}\right)\\
&=d\left(rm-m-\overline{\Delta}_2\right)-crm-qr\left(rm-m-\overline{\Delta}_2\right).
\end{align*}
Moreover, $m(r-1)-\overline{\Delta}_2>0$ for $r\geq 2$. This completes the proof of (b)
\end{proof}
\begin{remark}
If $\Delta_2=0$,  then  the first bound for $q$ in Theorem \ref{last2boundthm} will be $q\leq \frac{d-c}{r}$. Similarly, if $\overline{\Delta}_2=0$,  the second bound for $q$ will be $q\leq \frac{d}{r}-\frac{c}{r-1}$.
\end{remark}
\begin{remark} \label{remarkcruse} 
In Theorem \ref{last2boundthm}, let $\lambda=r=1, h=2$ (and so $c=m-1, d=n-1$). Since  $n\geq 2m-1$ and $n$ is even, we have  $n\geq 2m$. If $m$ is even (and so $\Delta_2=0$), we have $$\dfrac{d}{r}-\frac{cm}{rm-\Delta_2}=n-1-\frac{m(m-1)}{m}=n-m\geq m.$$
If $m$ is odd (and so $\Delta_2=1$), we have $$\dfrac{d}{r}-\frac{cm}{rm-\Delta_2}=n-1-\frac{m(m-1)}{m-1}=n-m-1\geq m-1.$$
Thus, a special case of Theorem \ref{last2boundthm} implies another result of Cruse \cite{MR0329925}  that a proper $(m-1)$-coloring of any subgraph of $K_m$ can be extended to a proper $(n-1)$-coloring of $K_n$ whenever $n\geq 2m$. 
\end{remark}

The proof of the next result is immediate from previous results of this section and we shall omit its proof. 
\begin{corollary}
Let $n \geq (h-1)(2m-1)$, $\cc F= \lambda K_n^h$,  $c=\lambda \binom{m-1}{h-1}$, $d=\lambda \binom{n-1}{h-1}$, $g=\dfrac{h}{\gcd(n,h)}$, and let $\cc G$ be an arbitrary sub-hypergraph of $\cc F$. 
We have
\begin{enumerate}[label=\textup{(\Roman*)}]
     \item  A proper $(d-c)$-coloring of $\cc G$  can be extended to a proper $d$-coloring of $\cc F$ whenever $n \equiv m\equiv 0 \Mod h$. 
    \item  A partial $g$-factorization of $\cc G$ using $\dfrac{d}{g}-\roundup{\dfrac{c}{g}}$ colors can be extended to a  $g$-factorization of $\cc F$ whenever  $gm \equiv 0 \Mod h$.
    \item  A partial $g$-factorization of $\cc G$ using $\dfrac{d}{g}-\roundup{\dfrac{c}{g-1}}$ colors can be extended to a connected $g$-factorization of $\cc F$ whenever $n \notequiv 0 \Mod h$, $m(g-1) \equiv 0 \Mod h$, and each color class of $\cc G$ is $g$-irregular. 
    \item  A partial $2$-factorization of $\cc G$ using $\dfrac{d}{2}-c$ colors can be extended to a connected $2$-factorization of $\cc F$ whenever $d$ is even, $m\equiv 2n \equiv 0 \Mod h$, and each color class of $\cc G$ is $2$-irregular. 
   \item  A partial $h$-factorization of $\cc G$ using $\rounddown{\dfrac{d}{h}}-\roundup{\dfrac{c}{h-1}}$ colors can be extended to a connected $h$-factorization of $\cc F$  whenever  $h\geq 2$,  $d \equiv 0 \Mod h$, $m \equiv 0 \Mod h$, and  each color class of $\cc G$ is $h$-irregular.
\end{enumerate}
\end{corollary}
\begin{remark} \textup{  (I) is an extension of the H\"{a}ggkvist-Hellgren theorem \cite{MR1249714}.  (II) and (III) generalize Baranyai's theorem \cite{MR0416986}. (IV) and (V) can be viewed as hypergraph analogues of Hilton's theorem on embedding path decompositions into Hamiltonian decompositions of complete graphs \cite{MR746544} }.\end{remark}

We conclude this section with a few examples.
\begin{example} \textup{
Any partial 1-factorization of any  sub-hypergraph of $K_8^3$ can be extended to an $r$-factorization of $K_{30}^3$ for each $r\in \{2,7,14,29,58\}$. Any partial 1-factorization of  any  sub-hypergraph of $K_8^4$ can be extended to an $r$-factorization of $K_{45}^4$ for each $r\in \{4,28,44,172,308\}$. Any partial 1-factorization of any  sub-hypergraph of $K_8^5$ can be extended to an $r$-factorization of $K_{60}^5$ for each $r\in \{2, 7, 14, 19, 29\}$. Finally, any partial 1-factorization of any  sub-hypergraph of $K_8^6$ can be extended to an $r$-factorization of $K_{75}^6$ for each $r\in \{2,4, 6, 12, 14\}$. In all these examples, we can ensure that each $r$-factor is connected.
}\end{example}
\begin{example} \label{examp2}\textup{ 
Any partial $6$-factorization of any sub-hypergraph of $K_{10}^3$ using 105 colors can be extended to a $6$-factorization of $K_{38}^3$.  Moreover, any partial $6$-factorization of any sub-hypergraph of $K_{10}^3$ using 103 colors in which each color class is $6$-irregular can be extended to a connected $6$-factorization of $K_{38}^3$. Here,  $k= 111$.
}\end{example}
\begin{example} \textup{ 
Any partial $23$-factorization of any sub-hypergraph of $K_{9}^3$ using 21 colors can be extended to a $24$-factorization of $K_{34}^3$.  Moreover, any partial $23$-factorization of any sub-hypergraph of $K_{9}^3$ using 20 colors can be extended to a connected $24$-factorization of $K_{34}^3$. In this example,  $k= 22$.
}\end{example}

\section{Concluding Remarks and Open Problems}
\begin{itemize}
    \item Using Newton's binomial theorem that for $r\in\mathbb R$ and $-1<x<1$, $$(1+x)^{r}=1+\sum_{i=1}^\infty \frac{r(r-1)\dots(r-i+1)}{i!}x^i,$$
one can slightly improve the bound $n\geq (h-1)(2m-1)$, but a significant improvement requires new ideas. Ryser \cite{MR42361} showed  that a partial 1-factorization of $K_{r,s}$ can be extended to a 1-factorization of $K_{n,n}$ if and only if 
\begin{align} \label{Ryserorigcond}
    E(K_{r,s}(i))\geq r+s-n \quad \forall i\in [n].
\end{align}
This condition is known as Ryser's condition. Unfortunately, the arithmetic obstructions (admissibility conditions) together with metric obstructions (the analogue Ryser-type condition) are not sufficient in our case. Suppose that a partial $\vb r$-factorization of $\lambda K_m^h$ is extended to an $\vb r$-factorization of $\lambda K_n^h$. The number of edges in each $r_i$-factor of $\lambda K_n^h$ is $r_i n/h$. Moreover, each vertex in $V(\lambda K_n^h)\backslash V(\lambda K_m^h)$ is incident with $r_i$ edges in $E(\lambda K_n^h)\backslash E(\lambda K_m^h)$. Therefore, $|E(\lambda K_m^h(i) )|+r_i(n-m)\geq r_i n /h$, and so the following Ryser-type condition is necessary. 
\begin{align} \label{rysergencond}
    |E(\lambda K_m^h(i) )| \geq r_i\left(m- n(1-\frac{1}{h})   \right) \quad \forall i\in [k].
\end{align}
In particular, we must have that $n\geq hm/(h-1)$ if there is at least one empty color class in $\lambda K_m^h$.  Depending on the initial coloring of $\lambda K_m^h$, there may be further necessary conditions. For example, one can also show that if the  coloring  of $\lambda K_m^h$ is an $r$-factorization and it is extended to an $r$-factorization of $\lambda K_n^h$, then $n\geq 2m$ is a necessary condition \cite{MR3910877} (for less trivial necessary conditions, see \cite{bahsadeurjc2021}). This leads us  to ask whether there are other partial $r$-factorizations of $\lambda K_m^h$ that require a stronger bound for $n$. In particular, answering the following question (even for $h=3$) will shed some light.
\begin{question}
Let $h\geq 3, n= \min \{ p > 2m \ |\  p \equiv 0 \Mod h  \}$. Does there exist a proper $\binom{n-1}{h-1}$-coloring of $K_m^h$ that cannot be extended to a proper $\binom{n-1}{h-1}$-coloring of $K_n^h$?
\end{question}

\item With respect to the main inequalities in Theorems \ref{incompembthm} and \ref{connthmcomb2}, it seems reasonable to think that we are doing the best one could hope for.  By example \ref{examp2}, any partial $6$-factorization of any sub-hypergraph of $K_{10}^3$ using 105 colors can be extended to a $6$-factorization of $K_{38}^3$. We conjecture that this is the best possible with respect to the number of colors.
\begin{conjecture}
There exists a  partial $6$-factorization of some sub-hypergraph of $K_{10}^3$ using 106 colors that cannot be extended to a $6$-factorization of $K_{38}^3$. 
\end{conjecture}

\item In our main embedding problem, a hole of size $m$ (edges of $\lambda K_m^h\subseteq \lambda K_n^h$) is colored, and we are entrusted to color the remaining edges of $\lambda K_n^h$. Complementary to this is the case where all the edges of $\lambda K_n^h$ are colored except for a hole of size $m$.
\begin{question}
Find conditions under which  a proper $\binom{n-1}{h-1}$-coloring of $K_n^h\backslash K_m^h$ can be extended to a proper $\binom{n-1}{h-1}$-coloring $K_n^h$. 
\end{question}  
  
\item Evans \cite{MR122728} asked if any proper $n$-coloring of any  $G\subseteq K_{n,n}$ with $|E(G)|=n-1$ can be extended to a proper $n$-coloring of $K_{n,n}$. This became a very popular problem in the 1970s, and Andersen and Hilton \cite{MR716801},    H\"{a}ggkvist \cite{MR519287}, and   Smetaniuk \cite{MR629869}  independently settled Evans' problem. Hence the following problem is natural. 
\begin{question}
For $n\equiv 0 \Mod h$,  what is the largest  $t:=t(n,h)$ such that any proper $\binom{n-1}{h-1}$-coloring of any $\cc G\subseteq K_n^h$ with $|E(\cc G)|=t$ can be extended to a proper $\binom{n-1}{h-1}$-coloring of $K_n^h$? Is it $n/h-1?$
\end{question}
Andersen and Hilton showed that $t(n,2)=n/2-1$ \cite{MR1276827}.

\item Last but not least,  enclosing decompositions \cite{MR3952136, MR3782232, MR3326169, MR2608434}, and highly edge-connected factorizations of complete graphs \cite{MR2325799, MR2071906} have been  studied. In the absence of any corresponding results for hypergraphs, we pose the following problems. 
\begin{question}
For $\mu > \lambda$ find conditions under which  a partial $\vb r$-factorization of $\lambda K_m^h$ can be extended to an $\vb r$-factorization of $\mu K_n^h$. 
\end{question}
\begin{question}
For $\ell \geq 2$ find conditions that ensure  a partial $\vb r$-factorization of $\lambda K_m^h$ can be extended to an $\ell$-edge-connected $\vb r$-factorization of $\lambda K_n^h$. 
\end{question}

\end{itemize}

\bibliographystyle{plain}

\end{document}